\documentclass[12pt]{article}

\usepackage[centertags]{amsmath}
\usepackage{amsfonts}
\usepackage{amssymb}
\usepackage{amsthm}
\usepackage{graphicx,caption, color}
\usepackage{verbatim}

\newtheorem{theo}{Theorem}
\newtheorem{lemma}[theo]{Lemma}

\newtheorem{prop}[theo]{Proposition}

\newcommand\be{\begin{equation}}
\newcommand\ee{\end{equation}}

\allowdisplaybreaks

\begin{document}
\title{On the $L^q$-dimensions of measures on \\ Hueter-Lalley type  self-affine sets}
\author{Jonathan M. Fraser$^1$ \& Tom Kempton$^2$ \\ \\
\emph{$^1$School of Mathematics,} \\ \emph{The University of Manchester,} \\ \emph{Manchester, M13 9PL, UK} \\ \emph{jon.fraser32@gmail.com}\\ \\
\emph{$^2$School of Mathematics,} \\ \emph{The University of St Andrews,} \\ \emph{St Andrews, KY16 9SS, UK} \\ \emph{tmwk@st-andrews.ac.uk}\\
 }
\maketitle
\abstract{\noindent  We study the $L^q$-dimensions of \emph{self-affine measures} and the \emph{K\"aenm\"aki measure} on a class of self-affine sets similar to the class considered by Hueter and Lalley. We give simple, checkable conditions under which the $L^q$-dimensions are equal to the value predicted by Falconer for a range of $q$. As a corollary this gives a wider class of self-affine sets for which the Hausdorff dimension can be explicitly calculated. Our proof combines the potential theoretic approach developed by Hunt and Kaloshin with recent advances in the dynamics of self-affine sets.\\ 



}

\setlength{\parskip}{.5cm}\setlength{\parindent}{0cm}

\section{Introduction}

Self-affine sets are natural objects of interest in fractal geometry, and have been intensively studied since the early 1980s. While the dimension theory of non-overlapping self-similar sets is well understood, significant open problems remain in the self-affine case. Indeed, until recently the only classes of self-affine sets for which it was possible to calculate the Hausdorff dimension were those with a `carpet structure' or those belonging to a class defined by Hueter and Lalley \cite{HuLal}.

Despite this lack of progress in the dimension theory of given self-affine sets and measures, the properties of typical self-affine sets and measures are much better understood.  Falconer defined the notion of {\it affinity dimension} and showed that for typical self-affine sets the affinity and Hausdorff dimensions coincide, provided the transformations are sufficiently contractive; see below for a formal statement.

In the last year new dynamical approaches have been developed to study the dimensions of self-affine sets, stemming from work of B\'ar\'any \cite{barany} and Falconer-Kempton  \cite{falconerkempton}. This work has been used to describe the Hausdorff dimension of a wide range of self-affine sets and measures \cite{rapaport, morrisshmerkin, baranykaenmaki}, as well as describing the Gibbs properties of natural self-affine measures, understanding projections of self-affine measures \cite{falconerkempton2}, and understanding the scenery flow \cite{kempton}.

In this article we turn our attention to the $L^q$-dimensions of self-affine measures; which provide finer information on the multifractal structure of the measure. The $L^q$-dimensions of self-affine measures have been studied intensively in various contexts, see \cite{kennethlq, kennethlq2, barralfeng,fengaffine, fraser, king, sponges, flexed}. Using dynamical systems defined in \cite{barany, falconerkempton} we give an explicit class of self-affine measures, based on the class defined by Hueter and Lalley, for which the $L^q$-dimensions can be computed for a range of $q$. We also give new results on the Hausdorff dimension of self-affine sets by extending the class of systems considered in \cite{HuLal}.

The main idea behind our proof goes as follows. We will be dealing with affine maps $T_i$ which contract different line segments by different amounts depending on the orientation of the line segment. First we express the problem of calculating $L^q$-dimensions in terms of bounding energy integrals, as is standard. We then ask how much a linear map $T$ contracts a line at angle $\theta$, and introduce a quantity $r^q_s(\theta)$ in Section \ref{mainproofsection} which averages over all concatenations $T_{\underline a}=T_{a_1}\circ\cdots\circ T_{a_n}$ the amount that $T_{\underline a}$ contracts lines at angle $\theta$. 

We show that proving $r^q_s(\theta)$ is uniformly bounded in $\theta$ implies that the relevant energy integrals are finite. Finally, using the dynamical systems of \cite{barany, falconerkempton} we give conditions under which the $r^q_s(\theta)$ are uniformly bounded, completing the proof. Our key idea here is to use `bunching conditions' and some degree of separation in a related iterated function system on projective space to show that, for any angle $\theta$, only a small proportion of the maps $T_{a_1\cdots a_n}$ are strongly contractive on lines at angle $\theta$.


\section{$L^q$-dimensions}

Let $\mu$ be a compactly supported Borel probability measure on $\mathbb{R}^d$.  The $L^q$-dimensions of $\mu$ give a coarse global description of the fluctuation of the measure on small scales and have many applications including in multifractal analysis and information theory.  For $q \geq 0 \, (q \neq 1)$ and $\delta>0$, let $\mathcal{M}_\delta$ be the set of closed $\delta$-mesh cubes imposed on $\mathbb{R}^d$ oriented and aligned with the coordinate axes and let
\[
M^q_\delta(\mu) \ = \ \sum_{Q \in \mathcal{M}_\delta} \mu(Q)^q,
\]
with the convention that $0^0 = 0$.  Then the upper and lower $L^q$-dimensions of $\mu$ are defined by
\[
\overline{D}^q(\mu) = \limsup_{\delta \to 0} \frac{\log M^q_\delta(\mu) }{(q-1) \log \delta}
\]
and
\[
\underline{D}^q(\mu) = \liminf_{\delta \to 0} \frac{\log M^q_\delta(\mu) }{(q-1) \log \delta}
\]
respectively.  If the two values coincide we denote the common value by $D^q(\mu)$ and refer to it as the $L^q$ dimension of $\mu$.  Note that $\underline{D}^0(\mu) $ and $\overline{D}^0(\mu) $ are the lower and upper box dimensions of the support of $\mu$ and $\underline{D}^2(\mu) $ and $\overline{D}^2(\mu) $ are the lower and upper correlation dimensions of $\mu$. We define $\overline{D}^1(\mu) $ and $\underline{D}^1(\mu) $ to be the upper and lower information dimensions of $\mu$, given by
\[
\overline{D}^1(\mu) = \limsup_{\delta \to 0} \frac{ \sum_{Q \in \mathcal{M}_\delta} \mu(Q) \log\mu(Q) }{ \log \delta}
\]
and
\[
\underline{D}^1(\mu) = \liminf_{\delta \to 0} \frac{ \sum_{Q \in \mathcal{M}_\delta} \mu(Q) \log\mu(Q) }{ \log \delta}
\]
respectively.  It is straightforward to see that the upper and lower $L^q$-dimensions are both non-increasing in $q \geq 0$ and continuous, except possibly at $q=1$. Therefore they are uniformly bounded above by $\overline{D}^0(\mu)$ and $\underline{D}^0(\mu)$, respectively, which are just the upper and lower box-counting dimensions of the support of $\mu$.

Lower bounds for the $L^q$-dimensions are given by energy integrals. It was proved in \cite[Proposition 2.1]{huntkaloshin} that if for some $q>1$ there exists $s=s(q)>0$ such that
\[
\mathcal{I}_s^q(\mu) \ := \ \int\left(\int \frac{d\mu(x)}{\lvert x-y \rvert^s}\right)^{q-1}d\mu(y) \ < \ \infty
\]
then $\underline{D}^q(\mu) \geq s$.

\subsection{Measures on self-affine sets}


For $\Lambda$ a finite set and $i\in\Lambda$ let $A_i$ be non-singular linear contractions on $\mathbb R^2$ (corresponding to $2\times 2$ matrices), and $t_i\in\mathbb R^2$ be translation vectors. Setting $T_i = A_i+t_i$, there is a unique non-empty compact set $F$ satisfying
\[
F = \bigcup_{i\in \Lambda} T_i(F).
\]
The set $F$ is called the \emph{self-affine} attractor of the IFS.  The \emph{singular values} of a $2 \times 2$ matrix $A$ are the positive square roots of the eigenvalues of $A^T A$.  Geometrically these numbers represent the lengths of the semi-axes of the image of the unit ball under $A$.  Thus the singular values correspond to how much the map contracts in different directions.  


For $\underline{a} \in \Sigma:=\Lambda^{\mathbb N}$, let $\underline{a}\vert_k \in \Lambda^k$ be the restriction of $\underline{a}$ to its first $k$ coordinates. We write
\[
T_{\underline a\vert_k} = T_{a_1} \circ \cdots \circ T_{a_k}.
\]


For $\underline a = a_1 \dots a_k \in \Lambda^k$ let  $1>\alpha_1(\underline a) \geq \alpha_2(\underline a)>0$ be the singular values of the linear part of $T_{\underline a}$. For $s \in [0,2]$ the \emph{singular value function} $\phi^s(\underline a)$ is given by
\[
\phi^s(\underline a) = \left\{ \begin{array}{cc}
\alpha_1^s(\underline a) & s \leq 1 \\
\alpha_1(\underline a)\alpha_2^{s-1}(\underline a) & s > 1 \\
\end{array} \right.
\]
Falconer defined the \emph{affinity dimension}, which depends only on the linear  parts of the maps in the IFS, by
\begin{equation} \label{affinitydim}
d = d (\{A_i\}_{i\in \Lambda}) = \inf \bigg\{ s: \sum_{k=1}^\infty \sum_{\underline{a} \in \Lambda^k}  \phi^s ( \underline{a} ) < \infty \bigg\}.
\end{equation}
It is always an upper bound for the Hausdorff and box dimensions of $F$, regardless of the choice of translations, and if one randomises the translations, then it is typically equal to the Hausdorff dimension provided the transformations $T_i$ are sufficiently contractive, see \cite{affine}.



Let $\mu$ be any Borel probability measure on $\Sigma$. We also let $\mu = \mu \circ \Pi^{-1}$ be the associated measure on the self-affine set $F$; whether we mean the symbolic or geometric $\mu$ will be clear from the context. In \cite{kennethlq} Falconer introduced a function $d(q)$ which is always an upper bound for the upper $L^q$-dimensions of a measure and in many cases gives the precise value.  This can be viewed as an extension of the of the affinity dimension which applies to measures on self-affine sets.

For $q>1$, let
\[
d(q) = \sup \left\{ s  \in \mathbb{R}^+ : \sum_{k=1}^\infty \sum_{\underline{a} \in \Lambda^k} \phi^s(\underline{a})^{1-q} \mu([\underline{a}])^q < \infty\right\}.
\]
Falconer \cite[Theorem 7.2]{kennethlq} shows that for all $q>1$ we have $\overline{D}^q(\mu) \leq \min \{d(q), 2\}$, provided that $\mu$ is a \emph{Gibbs measure}.  The measures we consider in this paper will be Gibbs, we refer the reader to \cite{kennethlq} for the definition of a Gibbs measure.

For a self-affine set, the `natural' measure from a dimension theoretic point of view is one which measures cylinders like the singular value function.  Thus the `weight' of a cylinder is close to its `geometric size'.  K\"aenm\"aki proved that such measures exist and, moreover, in our setting can be taken to be ergodic Gibbs measures.

\begin{prop}\label{muProp}
There exists an ergodic shift-invariant probability measure $\mu$ on $\Sigma$ and a constant $C_1$ such that
\[
\frac{1}{C_1} \phi^s(a_1\cdots a_n)\leq \mu[a_1\cdots a_n]\leq C_1  \phi^s(a_1\cdots a_n).
\]
Furthermore, there exists a constant $C_2$ such that
\[\frac{1}{C_2}\mu[a_1\cdots a_m]\mu[a_{m+1}\cdots a_n]\leq\mu[a_1\cdots a_n]\leq C_2\mu[a_1\cdots a_m]\mu[a_{m+1}\cdots a_n]
\]
for all sequences $\underline a\in\Sigma$ and $1\leq m\leq n$.
\end{prop}

The existence of this measure was proved by K\"aenm\"aki \cite{kaenmaki}, the fact that it is Gibbs was proved by  Barany and Rams \cite{BaranyRams}.

We also consider self-affine measures.  These are the push forwards of Bernoulli measures on $\Sigma$ via the coding map $\Pi$.  Alternatively, fix a positive probability vector $\{p(i)\}_{i \in \Lambda}$.  Then the associated self-affine measure is the unique Borel probability measure satisfying
\[
\mu = \sum_{i \in \Lambda} p(i) \, \mu \circ T_i^{-1}.
\]
Such measures are supported on $F$ but are typically singular with respect to the Hausdorff measure on $F$ and often exhibit a rich multifractal structure.



\section{Our setting and main results}

Given a $2\times 2$ matrix $A_i$ we let the corresponding projective linear transformation $\phi_i:\mathbb{PR}^1\to\mathbb{PR}^1$ be the map such that lines through the origin at angle $\theta$ are mapped to lines at angle $\phi_i(\theta)$ by the matrix $A_i^{-1}$. If $A_1,\cdots , A_k$ are strictly positive then the maps $\phi_i$ strictly contract the negative quadrant. 

Let $\{T_i\}_{i \in \Lambda}$ be an IFS of affine contractions acting on the plane  as described above and let $d$ be the affinity dimension associated with the linear parts. Also, let
\[
\gamma:= \exp \sup_{\theta\in \mathbb P\mathbb R^1} \limsup_{n\to\infty}\frac{1}{n}\log|\{c_1\cdots c_n:\theta\in\phi_{c_n\cdots c_1}(\mathcal Q_2)\}|.
\]
We will need the following assumptions:

(S) \textbf{Separation}: For all $i,j\in \Lambda$ with $i \neq j$ we have $T_i(F) \cap T_j(F) = \emptyset$.

(P) \textbf{Positivity}: The matrices $\{A_i\}_{i=1}^k$ are all strictly positive.

(B) \textbf{$(q-1)$-Bunching}:  For $q\geq 2$ the $(q-1)$-bunching condition is satisfied if for each  $i\in\Lambda$ one has
\[
\begin{array}{cc}
\gamma \alpha_1(i)^{qd} < \alpha_2(i)^{(q-1)d}
 \end{array}
\]
(MB) \textbf{$(q-1)$-Metric-Bunching}: If $\mu$ is a self-affine measure with weights $\{p(i)\}_{i \in \Lambda}$, then for $q\geq 2$ the $(q-1)$-metric-bunching condition is satisfied if for each  $i\in\Lambda$ one has
\[
\begin{array}{cc}
\gamma  \, p(i)^{q} < \alpha_2(i)^{(q-1)d(q)}.
 \end{array}
 \]

Although the definition of $\gamma$ is quite involved, it can be easily estimated to give conditions that are weaker but simple to check and state.  For example, $\gamma$ is always bounded above by the maximum number of first level intervals in the projective IFS which overlap a single point.  In many situations such as the class of self-affine sets studied by Hueter and Lalley, the IFS on projective space satisfies the appropriate version of strong separation (S) which renders $\gamma=1$.

The $(q-1)$-bunching condition becomes more restrictive as $q$ increases and if $q=2$ and $\gamma=1$, then our bunching condition returns the familiar \emph{1-bunched condition}.  

Finally, note that the conditions (S), (P) and (B) are conditions on the self-affine set, and make no reference to a particular measure.  Our results on the K\"aenm\"aki measure only rely on these conditions because the K\"aenm\"aki measure is chosen depending on the set.  The condition (MB) depends on the choice of self-affine measure.

\begin{theo}
Let $F$ be a self-affine set in the plane defined by an IFS satisfying Separation and Positivity, and let $\mu$ be the K\"aenm\"aki measure.  Assume that the affinity dimension $d \leq 1$ and that 
\[
q_0 := \sup\{q \geq 2 : \text{ the $(q-1)$-bunching condition is satisfied} \}
\]
is well-defined, i.e. the set of suitable $q \geq 2$ is nonempty. Then for all $q \in [0,q_0)$ we have
\[
D^q(\mu) = d = \dim_\mathrm{H} F.
\]
\end{theo}

In particular, setting $q=2$ in this theorem gives new conditions under which the Hausdorff and affinity dimension of the set $F$ agree.  These conditions are a generalisation of those of Hueter and Lalley since we allow overlaps in the IFS on projective space, which Hueter and Lalley do not allow.

We observe that $(q-1)$-bunching is an `open condition' (in $q$) and that the 1-bunched condition implies the $L^q$-dimensions are constantly equal to $d$ in an $[0, q_0)$ containing 2.  This is noteworthy because it implies the function $D^q(\mu)$ is differentiable in this interval, which is an important property in the context of dimension theory. For example, in general being differentiable at $q=1$ it guarantees that the measure is exact dimensional, see \cite{ngai}. 

\begin{theo}
Let $\mu$ be a self-affine measure in the plane defined by an IFS satisfying Separation and Positivity.   Also assume that
\[
q_0 := \sup\{q \geq 2 : \text{ the $(q-1)$-metric-bunching condition is satisfied} \}
\]
is well-defined, i.e. the set of suitable $q \geq 2$ is nonempty. Then for all $q \in [2,q_0)$ we have
\[
D^q(\mu) = d(q).
\]
\end{theo}

Again, note that $(q-1)$-metric-bunching is an `open condition' in $q$.  It is of interest to determine situations where our bunching conditions hold.  In the case of Bernoulli measures, one may turn (MB) into a condition on the set by asking for which self-affine sets does there exist a self-affine measure which satisfies (MB)?  This condition is easily seen to be equivalent to
\[
\sum_{i \in \Lambda} \left(\frac{\alpha_2(i)^{(q-1)d(q)}}{\gamma}\right)^{1/q} > 1.
\]
This is an easily checked condition, and in the setting of Hueter-Lalley it is always satisfied for some $q_0$.

\begin{prop}
Let $F$ be a planar self-affine set satisfying Separation and Positivity.  Also assume that $\gamma=1$, $d \leq 1$ and the 1-bunching condition is satisfied.  Then there exists a self-affine measure $\mu$ on $F$ and some $q_0>2$ such that for all $q \in [2,q_0)$ we have
\[
D^q(\mu) = d(q).
\]
\end{prop}

\begin{proof}
Given the previous discussion, it suffices to show that 
\[
\sum_{i \in \Lambda} \alpha_2(i)^{d(2)/2} > 1.
\]
Indeed, the 1-bunching condition and the fact that $d \geq d(2)$ guarantees
\[
\sum_{i \in \Lambda} \alpha_2(i)^{d(2)/2} > \sum_{i \in \Lambda} \alpha_1(i)^{d(2)} \geq \sum_{i \in \Lambda} \alpha_1(i)^{d}
\]
and using submultiplicativity of the larger singular value function, this is at least 1 by the definition of $d$.
\end{proof}

\section{Proofs}

We will first prove the result for the K\"aenm\"aki measure.  We will assume $q \geq 2$ and observe that this is sufficient to prove the result for all $q \in [0,q_0)$ by virtue of the fact that the $L^q$-dimensions are non-increasing in $q$ and always bounded above by the box dimension and hence the affinity dimension of $F$.  This means that establishing $D^2(\mu) = d$ is enough to prove that $D^q(\mu) = d$ for $q \in [0,2)$.

Finally, in Section \ref{bernoulliproof} we adapt our proof of the K\"aenm\"aki measure case to give the result for general Bernoulli measures.

\subsection{The case: $q \geq 2$}\label{mainproofsection}

Let $\mu$ be the K\"aenm\"aki  measure and let $0<s<d$.  Let us consider the energy integral
\[
\mathcal I_{s}^q(\mu)=\int_{\underline a \in \Sigma}\left(\int_{\underline b \in\Sigma} \frac{1}{d(\underline a,\underline b)^s}d\mu(\underline b)\right)^{q-1}d\mu(\underline a)
\]
where $d(\underline a,\underline b)$ denotes the distance between the points $\pi(\underline a),\pi(\underline b)\in F$ coded by $\underline a$ and $\underline b$ respectively. We wish to show that $I_s^q(\mu)<\infty$ for $s<d(q)$.

\begin{lemma}\label{keyestimate}
Suppose that
\[
\iint_{\underline a, \underline b : \underline a \wedge \underline b = \phi} \sum_{n=0}^{\infty}\sum_{c_1\cdots c_n} \frac{\mu[c_1\cdots c_n]^q}{d( c_1\cdots c_n\underline a,  c_1\cdots c_n\underline b)^{s(q-1)}}d\mu(\underline a)d\mu(\underline b)<\infty
\]
for all $s<d(q)$. Then it follows that $I_s^q(\mu)<\infty$ for $s<d(q)$.
\end{lemma}

The proof of this lemma is a little delicate, but it is distinct from the main new ideas of this article and so is postponed until Section \ref{technicallemma}.

Observe that the linear map $T_{c_1}\circ \cdots \circ T_{c_n}$ contracts straight line segments by an amount depending only on the angle of the straight line segment. Let 
\[\lambda_{c_1\cdots c_n}(\theta):= \frac{d(c_1\cdots c_n\underline a,c_1\cdots c_n\underline b)}{d(\underline a,\underline b)}\]
for any $\underline a,\underline b$ with $\theta=\theta(\underline a,\underline b)$ defined to be the angle of the line connecting $\pi(\underline a)$ and $\pi(\underline b)$. This is well defined and depends only on $\theta(\underline a,\underline b)$ rather than the sequences $\underline a$ and $\underline b$ themselves.

Define a function $r_s^q:\mathbb P\mathbb R^1\to \mathbb R$  by
\[
r_s^q(\theta):=\sum_{n=0}^{\infty} \sum_{c_1\cdots c_n}(\mu[c_1\cdots c_n])^q(\lambda_{c_1\cdots c_n}(\theta))^{-s(q-1)}.
\]
The separation assumption (S) guarantees that the number
\[
m:= \inf \{d(\underline a,\underline b) : \underline a,\underline b\in\Sigma : \underline a \wedge\underline b=\phi\}
\]
is strictly positive. We can rewrite the integral in Lemma \ref{keyestimate} as
\begin{align*}
& \hspace{-10mm}  \int\int_{\underline a,\underline b\in\Sigma : \underline a \wedge\underline b=\phi} \frac{1}{d(\underline a,\underline b)^{s(q-1)}}r_s^q(\theta(\underline a,\underline b))d\mu(\underline a)d\mu(\underline b)\\
&\leq \frac{1}{m^{s(q-1)}}\int\int_{\underline a,\underline b\in\Sigma : \underline a \wedge\underline b=\phi} r_s^q(\theta(\underline a,\underline b))d\mu(\underline a)d\mu(\underline b).
\end{align*}
Thus our goal is now to show that  $r_s^q(\theta)$ is uniformly bounded under our conditions, which in turn shows that the energy integral $\mathcal I_{s}^q(\mu)$ is finite and thus by letting $s \to d$ we conclude that $\underline{D}^q(\mu) \geq d$.  Combined with Falconer's upper bound and the observation that for the K\"aenm\"aki measure $d(q) = d$ for all $q\geq 0$, this proves that $D^q(\mu)  = d$ as required.

\subsubsection{Bounding  $r_s^q(\theta)$ uniformly}

We begin by noting a couple of facts about the contraction rates $\lambda_{c_1\cdots c_n}$. Firstly, for $\theta\in\mathbb{PR}^1, \underline a\in\Sigma$ and $1\leq m\leq n$ we have
\begin{align*}
\lambda_{c_n\cdots c_1}(\theta)&=\frac{d(c_n\cdots c_1\underline a, c_n\cdots c_1\underline b)}{d(c_m\cdots c_1\underline a,c_m\cdots c_1\underline b)}\frac{d(c_m\cdots c_1\underline a,c_m\cdots c_1\underline b)}{d(\underline a,\underline b)}\\
&= \lambda_{c_n\cdots c_{m+1}}(\phi_{c_1\cdots c_m}^{-1}(\theta))\lambda_{c_m\cdots c_1}(\theta).
\end{align*}
Here we used that
\[\phi_{c_1\cdots c_m}^{-1}=(\phi_{c_1}\circ\cdots \circ \phi_{c_m})^{-1}=\phi_{c_m}^{-1}\circ\cdots\circ \phi_{c_1}^{-1}.
\]
Now let $D_{c_1\cdots c_n}$ denote the $n$th level ellipse coded by $c_1\cdots c_n$ and note that
\[
\alpha_2(c_1\cdots c_n)\leq \lambda_{c_1\cdots c_n}(\theta)\leq \alpha_1(c_1\cdots c_n).
\]
When $\theta$ is chosen to be the angle of the minor (resp. major) axis of $D_{c_1\cdots c_n}$ then the left hand (resp. right hand) inequality becomes an equality. 

\begin{lemma}\label{C3}
There exists a constant $C_3>1$ such that
\[
(\alpha_1(a_1\cdots a_n))^{-s}\leq (\lambda_{a_1\cdots a_n}(\theta))^{-s}\leq C_3 (\alpha_1(a_1\cdots a_n))^{-s}
\]
for all $\theta\in\mathcal Q_1, a_1\cdots a_n\in\Lambda^n$.
\end{lemma}
\begin{proof}
Since each of our matrices $A_i$ is strictly positive, the directions of the long axes of the ellipses $D_{a_1\cdots a_n}$ all lie in some cone contained strictly inside $\mathcal Q_1$. In particular, there exists some closed interval $K\subset (-\frac{\pi}{2},\frac{\pi}{2})$ such that the angle between the long axis of $D_{a_1\cdots a_n}$ and $\theta$ lies in $K$ for any $\theta\in\mathcal Q_1$.

Then, since $\lambda_{a_1\cdots a_n}(\theta)\geq \alpha_1(a_1\cdots a_n)\cos(\theta')$ where $\theta'$ is the angle between $\theta$ and the major axis of $D_{a_1\cdots a_n}$, we are done.
\end{proof}

We split words $a_n\cdots a_1$ into two parts according to how close $a_1\cdots a_n$ is to being a code of $\theta$. More precisely, we write $a_n\cdots a_1=e_n\cdots e_{m+1}c_m\cdots c_1$ where $m$ is the largest integer for which $\theta\in\phi_{a_1\cdots a_m}(\mathcal Q_2)$. Then
\begin{align*}
r_s^q(\theta)& = \sum_{n=0}^{\infty}\sum_{a_n\cdots a_1}(\mu[a_n\cdots a_1])^q(\lambda_{a_n\cdots a_1}(\theta))^{-s(q-1)} \\
&\leq C_2^{2q}\sum_{m=0}^{\infty}\sum_{c_1\cdots c_m \text{ coding } \theta} \sum_{n=m+1}^{\infty} \sum_{e_{m+1}\cdots e_n} \\
&\,  (\lambda_{c_m\cdots c_1}(\theta))^{-(q-1)s}(\lambda_{e_{n}\cdots e_{m+1}}(\phi_{c_m\cdots c_1}^{-1}(\theta))^{-(q-1)s}(\mu[c_m\cdots c_1])^q(\mu[e_{n}\cdots e_{m+1}])^q\\
& \leq C_2^{2q}C_3^{q-1} \sum_{m=0}^{\infty}\sum_{c_1\cdots c_m \text{ coding } \theta} \sum_{n=m+1}^{\infty} \sum_{e_{m+1}\cdots e_n}  \\
&\,  (\alpha_2(c_m\cdots c_1))^{-(q-1)s}(\alpha_1(e_{n}\cdots e_{m+1}))^{-(q-1)s}(\mu[c_m\cdots c_1])^q(\mu[e_{n}\cdots e_{m+1}])^q.
\end{align*}
Here the first inequality came from the quasi-Bernoulli property of $\mu$, see Proposition \ref{muProp}. For the second inequality we used the fact that  $\lambda_{c_m\cdots c_1}(\theta)\geq \alpha_2(c_m\cdots c_1)$ to deal with the terms $c_m\cdots c_1$ corresponding to codings of $\theta$. Finally we used Lemma \ref{C3} and the fact that the angles $\phi_{e_k\cdots e_{m+1}c_m\cdots c_1}^{-1}(\theta)$ are not in $\mathcal Q_2$.

Dropping constants (as we may) the right hand side of the above inequality can be factorised, giving 
\begin{align*}
r_s^q(\theta)&\leq\sum_{m=0}^{\infty}\sum_{c_1\cdots c_m \text{ coding } \theta}(\alpha_2(c_m\cdots c_1))^{-(q-1)s}(\mu[c_m\cdots c_1])^q \\
&\, \qquad \sum_{n=m+1}^{\infty} \sum_{e_{m+1}\cdots e_{n}}(\alpha_1(e_{n}\cdots e_{m+1}))^{-(q-1)s}(\mu[e_{n}\cdots e_{m+1}])^q\\
&=\left(\sum_{m=0}^{\infty}\sum_{c_1\cdots c_m \text{ coding } \theta}(\alpha_2(c_m\cdots c_1))^{-(q-1)s}(\mu[c_m\cdots c_1])^q\right)\\
&\, \qquad  \times\left(\sum_{j=1}^{\infty} \sum_{e_{1}\cdots e_j}(\alpha_1(e_{1}\cdots e_j))^{-(q-1)s}(\mu[e_{1}\cdots e_j])^q\right)
\end{align*}
The second equality here is just a relabelling of $e_{m+1}\cdots e_n$ since the quantity $n$ is no longer relevant. The two terms in this multiplication are independent. The second term will be easily bounded by applying the definition of $d(q)$ and the first term will be bounded using our bunching condition. 

First observe that the second term is always uniformly bounded.  Since $s < d \leq 1$, $\alpha_1(\underline{a})^s = \phi^s(\underline{a})$ for any finite word $\underline{a}$, and so
\begin{align*}
\sum_{j=1}^{\infty} \sum_{e_{1}\cdots e_j}(\alpha_1(e_{1}\cdots e_j))^{-(q-1)s}(\mu[e_{1}\cdots e_j])^q& = \sum_{j=1}^{\infty} \sum_{e_{1}\cdots e_j}\phi^s(e_1 \cdots e_j)^{1-q}(\mu[e_{1}\cdots e_j])^q
\end{align*}
which is finite precisely when $s<d(q) = d$, by the definition of $d(q)$.  The remainder of the argument comes down to bounding the first term.  This is where we need more restrictive assumptions on the measure in the form of bunching conditions.  Using $s  < d \leq 1$ we have

\begin{align*}
\sum_{m=0}^{\infty}\sum_{c_1\cdots c_m \text{ coding } \theta}(\alpha_2(c_m\cdots c_1))^{-(q-1)s}(\mu[c_m\cdots c_1])^q \\
 &  \hspace{-60mm}\leq  \sum_{m=0}^{\infty}\sum_{c_1\cdots c_m \text{ coding } \theta}(\alpha_2(c_m\cdots c_1))^{-(q-1)s} C_2^q \alpha_1(c_m\cdots c_1)^{dq}  \\
& \hspace{-60mm}\leq C_2^q \sum_{m=0}^{\infty}\sum_{c_1\cdots c_m \text{ coding } \theta} \prod_{k=1}^m \frac{ \alpha_1(c_k)^{dq}}{\alpha_2(c_k)^{(q-1)s}}   \\
& \hspace{-60mm}\leq C_2^q \sum_{m=0}^{\infty} \left( \gamma \, \max_{i \in \Lambda}\frac{ \alpha_1(i)^{dq}}{\alpha_2(i)^{(q-1)s}}\right)^m  
\end{align*}

This is a geometric series which sums whenever $\gamma \alpha_{1}(i)^{d q} < \alpha_{2}(i)^{(q-1)s}$ for all $i \in \Lambda$, which is guaranteed by our bunching condition since $s< d$.

\subsubsection{Why we are restricted to $d \leq 1$?}

In this section we discuss the case where the affinity dimension is strictly larger than 1. Our techniques seem to generalise quite naturally under the assumption of a slightly stronger bunching condition however, unfortunately, this bunching condition is never satisfied, and so our techniques do not yield information about the dimension of measures of Hausdorff dimension greater than one.

The argument is similar to the previous section until the point where we are led to bound two independent sums.  We only need to consider the `difficult' first term, since the second term is also easily shown to be bounded in this case.  Assuming $2 \geq d>s > 1$ we have
\begin{align*}
\sum_{m=0}^{\infty}\sum_{c_1\cdots c_m \text{ coding } \theta}(\alpha_2(c_m\cdots c_1))^{-(q-1)s}(\mu[c_m\cdots c_1])^q \\
 &  \hspace{-80mm}\leq  \sum_{m=0}^{\infty}\sum_{c_1\cdots c_m \text{ coding } \theta}(\alpha_2(c_m\cdots c_1))^{-(q-1)s} C_2^q \alpha_1(c_m\cdots c_1)^{q} \alpha_2(c_m\cdots c_1)^{q(d-1)} \\
& \hspace{-80mm}\leq C_2^q \sum_{m=0}^{\infty}\sum_{c_1\cdots c_m \text{ coding } \theta} \prod_{k=1}^m \frac{ \alpha_1(c_k)^{q}}{\alpha_2(c_k)^{q-s}}   \\
& \hspace{-80mm}\leq C_2^q \sum_{m=0}^{\infty} \left( \gamma \, \max_{i \in \Lambda}\frac{ \alpha_1(i)^{q}}{\alpha_2(i)^{q-s}}\right)^m  
\end{align*}

This is a geometric series which sums whenever $\gamma \alpha_{1}(i)^{q}  \alpha_{2}(i)^{s-q} < 1$ for all $i \in \Lambda$.  Although this condition looks natural, curiously it appears to be vacuous by the following heuristic reasoning.

First note that the maps $\phi_{\underline{a}}$ contract the negative quadrant of projective space by $\frac{\alpha_2(\underline{a})}{\alpha_1(\underline{a})}$ (up to constants). A good $k$th level estimate for $\gamma$ is given by counting the the maximum number of mutually overlapping intervals in our IFS on projective space, which by the pigeon hole principal, gives
\[
\gamma^k \geq \sum_{\underline{a} \in \Lambda^k}\frac{\alpha_2(\underline{a})}{\alpha_1(\underline{a})},
\]
again up to constants.  Therefore a good estimate for $\gamma$ is
\[
\liminf_{k \to \infty} \left(\sum_{\underline{a} \in \Lambda^k}\frac{\alpha_2(\underline{a})}{\alpha_1(\underline{a})}\right)^{1/k}
\]
which yields
\[
1 \geq \liminf_{k \to \infty} \left(\sum_{\underline{a} \in \Lambda^k}\gamma^{-k} \frac{\alpha_2(\underline{a})}{\alpha_1(\underline{a})}\right)^{1/k}.
\]
Recall the `bunching condition' required above, with $q=2$: $ \gamma \alpha_{1}(i)^{2}  \alpha_{2}(i)^{s-2} < 1$ (for all $i \in \Lambda$).  Using sub- and super-additivity of the larger and smaller singular values respectively  implies that
\[
\alpha_{1}(\underline{a})^{2}  \alpha_{2}(\underline{a})^{s-2} \leq \prod_{l=1}^{k} \frac{\alpha_1(a_l)^2}{ \alpha_{2}(a_l)^{2-s}} < \gamma^{-k}
\]
and combining these two estimates yields
\[
1 \geq \liminf_{k \to \infty} \left(\sum_{\underline{a} \in \Lambda^k}  \alpha_1(\underline{a}) \alpha_2(\underline{a})^{s-1}\right)^{1/k}
\]
which implies $s \geq d$ which is a contradiction.

\subsection{Proof of Lemma \ref{keyestimate}} \label{technicallemma}
Recall that
\[
\mathcal I_{s}^q(\mu):=\int_{\underline a \in \Sigma}\left(\int_{\underline b \in\Sigma} \frac{1}{d(\underline a,\underline b)^s}d\mu(\underline b)\right)^{q-1}d\mu(\underline a).
\]

We rewrite the term $\underline b$ (in the inner integral) in terms of a part that agrees with $\underline a$ (in the outer integral) and a part that differs, setting $\underline b=a_1\cdots a_n\underline b'$ for the longest word $a_1\cdots a_n$ possible. This gives

\[
\mathcal I_s^q(\mu)=\int_{\underline a \in \Sigma}\left(\sum_{n=0}^{\infty}\int_{\underline b' \in\Sigma:\underline b' \wedge \sigma^n\underline a = \phi} \frac{\mu[a_1\cdots a_n]}{d(\underline a,a_1\cdots a_n\underline b')^s}d\mu(\underline b')\right)^{q-1}d\mu(\underline a)
\]

It is sufficient to consider integrals with some extra decay, which are easier to manipulate using Jensen's inequality. For $\beta \in (0,1)$ define
\begin{eqnarray*}
\mathcal I_s^{\beta,q}(\mu) &:=& \int_{\underline a \in \Sigma}\left(\dfrac{\sum_{n=0}^{\infty}\beta^n\int_{\underline b' \in\Sigma:\underline b' \wedge \sigma^n\underline a = \phi} \frac{\mu[a_1\cdots a_n]}{d(\underline a,a_1\cdots a_n\underline b')^{s'}}d\mu(\underline b')}{\sum_{n=1}^{\infty}\beta^n}\right)^{q-1}d\mu(\underline a).
\end{eqnarray*}
\begin{lemma}
Suppose that $\mathcal I_{s'}^{\beta,q}(\mu)<\infty$ for all $s'<d(q), \beta<1$. Then $\mathcal I_s^q(\mu)<\infty$ for all $s<d(q)$.
\end{lemma}

\begin{proof} Let $s<d(q)$ and assume without loss of generality that the diameter of $F$ is 1.  For $s'\in (s,d(q))$ choose $\beta$ such that
\[
\left(\max_i \alpha_1(i) \right)^{s'-s} \leq  \beta < 1.
\]
Then for any $\sigma^n(\underline a)\wedge\underline b'=\phi$ and $n$ large we have
\begin{eqnarray*}
\frac{\beta^n}{d(\underline a,a_1\cdots a_n\underline b')^{s'}}&=&\left(\frac{\beta^n}{d(\underline a,a_1\cdots a_n\underline b')^{s'-s}}\right)\frac{1}{d(\underline a,a_1\cdots a_n\underline b')^s} \\ \\
&\geq&\left(\frac{\beta^n}{(\max_i \alpha_1(i)^n )^{s'-s}}\right)\frac{1}{d(\underline a,a_1\cdots a_n\underline b')^s} \\ \\
&\geq& \frac{1}{d(\underline a,a_1\cdots a_n\underline b')^s}.
\end{eqnarray*}
Integrating over all $\underline a, \underline b$, we have shown that
\[
\mathcal I_s^q(\mu) \ \leq  \ \left(\sum_{n=1}^{\infty}\beta^n\right)^{q-1}  \mathcal I_{s'}(\mu)^{\beta,q} \ < \ \infty,
\]
completing the proof of the lemma. 
\end{proof}
Now we have
\begin{eqnarray*}
& &\left(\dfrac{\sum_{n=0}^{\infty}\beta^n\int_{\underline b' \in\Sigma:\underline b' \wedge \sigma^n\underline a = \phi} \frac{\mu[a_1\cdots a_n]}{d(\underline a,a_1\cdots a_n\underline b')^{s'}}d\mu(\underline b')}{\sum_{n=1}^{\infty}\beta^n}\right)^{q-1}\\
&<&\dfrac{\sum_{n=0}^{\infty}\beta^n\left(\int_{\underline b' \in\Sigma:\underline b' \wedge \sigma^n\underline a = \phi} \frac{\mu[a_1\cdots a_n]}{d(\underline a,a_1\cdots a_n\underline b')^{s'}}d\mu(\underline b')\right)^{q-1}}{\sum_{n=1}^{\infty}\beta^n}
\end{eqnarray*}
using Jensen's inequality in the form $\left(\frac{\sum a_i x_i}{\sum a_i}\right)^{q-1}\leq \frac{\sum a_i (x_i)^{q-1}}{\sum a_i}$ for $q>2$. 

Now using the integral form of Jensen's inequality, we move the $q-1$ above inside the integral, giving
\[
\mathcal I_s^{\beta,q}(\mu)\leq C'\int_{\underline a\in\Sigma}\dfrac{\sum_{n=0}^{\infty}\beta^n\int_{\underline b' \in\Sigma:\underline b' \wedge \sigma^n\underline a = \phi} \left(\frac{\mu[a_1\cdots a_n]}{d(\underline a,a_1\cdots a_n\underline b')^{s'}}\right)^{q-1}d\mu(\underline b')}{\sum_{n=1}^{\infty}\beta^n}d\mu(\underline a)
\]
for some universal constant $C'>0$.  If the integrals here were over sets of measure $1$, Jensen's inequality would work precisely and the constant $C'$ would not be necessary. We have $\mu\{\underline b':\sigma^n\underline a\wedge \underline b'=\phi\}<1$ but since the measure of these sets is uniformly bounded below we can renormalise at the expense of a universal constant. In particular, for any $n$ we have
\[
\mu\{\underline b':\sigma^n\underline a\wedge \underline b'=\phi\} \geq  1- \max_{i \in \Lambda} \mu[i]
\]
which is uniformly bounded away from 0, using the Gibbs property.

Since we are interested only in showing finiteness, without loss of generality we may discard the denominator and the constant $C'$. We write $\underline a'=\sigma^n \underline a$ and reorder integration to get
\begin{eqnarray*}
 \mathcal I_s^{\beta,q}(\mu)&<&\sum_{n=0}^{\infty}\beta^n\sum_{a_1\cdots a_n}\mu[a_1\cdots a_n] \\ \\
&\,& \hspace{-20mm}  \int_{\underline a'\in\Sigma}\int_{\underline b'\in\Sigma:\underline a'\wedge\underline b'=\phi}\frac{\mu[a_1\cdots a_n]^{q-1}}{d(a_1\cdots a_n\underline a', a_1\cdots a_n \underline b')^{s(q-1)}} d\mu(\underline b')d\mu(\underline a').
\end{eqnarray*}
Finally, discarding $\beta$, since it only decreases things, using Fubini's theorem and renaming $a_1\cdots a_n$ as $c_1\cdots c_n$ gives

\[
\mathcal I_s^{\beta,q}(\mu)\leq \iint_{\underline a, \underline b : \underline a \wedge \underline b = \phi} \sum_{n=0}^{\infty}\sum_{c_1\cdots c_n} \frac{\mu[c_1\cdots c_n]^q}{d( c_1\cdots c_n\underline a,  c_1\cdots c_n\underline b)^{s(q-1)}}d\mu(\underline a)d\mu(\underline b).
\]
Since it is enough to show that $\mathcal I_s^{\beta,q}(\mu)$ is finite, the proof of Lemma \ref{keyestimate} is complete.

\subsection{Bernoulli measures} \label{bernoulliproof}

In this section we let $\mu$ be a Bernoulli measure (a self-affine measure). Bernoulli measures clearly satisfy the quasi-Bernoulli property, and so the vast majority of the previous argument also goes through in this setting.  Note that the `simple term' is bounded for any $s< d(q)$.  The only real difference comes when trying to bound the `difficult term', which involved bunching conditions and properties of the measure.  In particular, we must bound
\[
\sum_{m=1}^{\infty}\sum_{c_1\cdots c_m \text{ coding } \theta}(\alpha_2(c_m\cdots c_1))^{-(q-1)s}(\mu[c_m\cdots c_1])^q.
\]
Here the behaviour is the same for any value of $s \in (0,2]$ and a similar argument yields:
\begin{align*}
\sum_{m=0}^{\infty}\sum_{c_1\cdots c_m \text{ coding } \theta}(\alpha_2(c_m\cdots c_1))^{-(q-1)s}(\mu[c_m\cdots c_1])^q \\
 &  \hspace{-60mm} =  \sum_{m=0}^{\infty}\sum_{c_1\cdots c_m \text{ coding } \theta}(\alpha_2(c_m\cdots c_1))^{-(q-1)s} \left( p(c_m)\cdots p(c_1) \right)^{q}  \\
& \hspace{-60mm}\leq  \sum_{m=0}^{\infty}\sum_{c_1\cdots c_m \text{ coding } \theta} \prod_{k=1}^m \frac{ p(c_k)^{q}}{\alpha_2(c_k)^{(q-1)s}}   \\
& \hspace{-60mm}\leq\sum_{m=0}^{\infty} \left( \gamma \, \max_{i \in \Lambda}\frac{ p(i)^{q}}{\alpha_2(i)^{(q-1)s}}\right)^m  
\end{align*}

This is a geometric series which sums whenever $\gamma \, p(i)^{q} < \alpha_{2}(i)^{(q-1)s}$ for all $i \in \Lambda$, which is guaranteed by our metric bunching condition, provided $s<d(q)$.  The rest of the argument goes through as before and is omitted.

\begin{centering}

\textbf{Acknowledgements}\\
The authors were financially supported by an LMS Scheme 4 \emph{Research in Pairs} grant. TK also acknowledges financially support from the EPSRC grant EP/K029061/1. The authors thank the Universities of Manchester and St Andrews for hosting the research visits which led to this work and Kenneth Falconer and Antti K\"aenm\"aki for helpful discussions. 

\end{centering}

\end{document}